\documentclass[12pt]{amsart}
\usepackage{amsmath,amssymb,amsthm,graphics,amscd,mathrsfs}
\usepackage{amscd, amsfonts, amssymb, amsthm}
\usepackage{float, fullpage, verbatim, bbm}
\usepackage[colorlinks, breaklinks, linkcolor=blue]{hyperref}
\usepackage{breakurl}
\usepackage{array}
\usepackage{latexsym}
\usepackage{enumerate}

\usepackage{xcolor}
\usepackage{subcaption}
\usepackage{tikz-cd}
\usepackage{tikz}
\usetikzlibrary{arrows}
\usetikzlibrary{shapes}
\usepackage{tkz-graph}
\usetikzlibrary{decorations}
\usetikzlibrary{arrows,decorations.markings}
\usetikzlibrary{calc}
\usepackage{multirow}
\tikzstyle{v} = [circle, draw, inner sep=2pt, minimum size=3pt, fill=black]
\tikzstyle{l} = [rectangle, draw, rounded corners]

\usepackage{etoolbox}
\usepackage{caption}
\usepackage[parfill]{parskip}

\newcommand\CA{{\mathscr A}}
\newcommand\CB{{\mathscr B}}
\newcommand\CC{{\mathscr C}}
\renewcommand\CD{{\mathscr D}}

\newcommand\CH{{\mathscr H}}
\newcommand\CS{{\mathcal S}}
\newcommand\CIF{{\mathcal {IF}}}

\newcommand\CIFAC{{\mathcal {IF\!AC}}}

\newcommand\CI{{\mathcal I}}

\newcommand\R{{\varrho}}
\newcommand\RR{{\mathscr R}}

\newcommand\BBC{{\mathbb C}}
\newcommand\BBK{{\mathbb K}}
\newcommand\BBN{{\mathbb N}}

\newcommand\BBQ{{\mathbb Q}}
\newcommand\BBR{{\mathbb R}}
\newcommand\BBZ{{\mathbb Z}}

\newcommand\codim{\operatorname{codim}}

\newcommand\GL{\operatorname{GL}}

\newcommand\rank{\operatorname{rank}}
\newcommand\rk{\operatorname{rk}}

\newcommand\Gen{\operatorname{Gen}}
\newcommand\lc{{\operatorname{lc}}}

\numberwithin{equation}{section}

\theoremstyle{plain}
\newtheorem{lemma}[equation]{Lemma}
\newtheorem{theorem}[equation]{Theorem}

\newtheorem{conjecture}[equation]{Conjecture}

\newtheorem{corollary}[equation]{Corollary}
\newtheorem{prop}[equation]{Proposition}
\theoremstyle{definition}
\newtheorem{defn}[equation]{Definition}
\newtheorem{remark}[equation]{Remark}


\begin{document}

\title[Hyperpolygonal arrangements]
{Hyperpolygonal arrangements}

\author[L. Giordani]{Lorenzo Giordani}
\address
{Fakult\"at f\"ur Mathematik,
	Ruhr-Universit\"at Bochum,
	D-44780 Bochum, Germany}
\email{lorenzo.giordani@rub.de}

\author[P.~M\"ucksch]{Paul M\"ucksch}
\address
{Technische Universität Berlin,
	Institut für Mathematik,
	Einsteinufer 19,
	D-10587 Berlin, Germany}
\email{paul.muecksch+uni@gmail.com}

\author[G. R\"ohrle]{Gerhard R\"ohrle}
\address
{Fakult\"at f\"ur Mathematik,
Ruhr-Universit\"at Bochum,
D-44780 Bochum, Germany}
\email{gerhard.roehrle@rub.de}

\author[J. Schmitt]{Johannes Schmitt}
\address
{Fakult\"at f\"ur Mathematik,
	Ruhr-Universit\"at Bochum,
	D-44780 Bochum, Germany}
\email{johannes.schmitt@rub.de}


\allowdisplaybreaks

\begin{abstract}
In \cite{bellamyetall:crepant},
a particular family of real hyperplane arrangements stemming from  
hyperpolygonal spaces associated with certain 
quiver varieties was introduced which we thus call 
\emph{hyperpolygonal arrangements} $\CH_n$.
In this note we study these
arrangements and 
investigate their properties systematically. 
Remarkably the arrangements 
$\CH_n$ discriminate between essentially all local properties of arrangements.
In addition we show that hyperpolygonal arrangements are projectively unique and combinatorially formal.

We note that the arrangement $\CH_5$
is the famous counterexample of Edelman and Reiner \cite{edelmanreiner:orlik} of Orlik's conjecture that the restriction of a free arrangement is again free. 
\end{abstract}

\maketitle


\section{Introduction and Main Results}

In \cite{bellamyetall:crepant},
Bellamy et al
introduced a particular family of real hyperplane arrangements in $\BBR^n$ stemming from 
hyperpolygon spaces realized as certain 
quiver varieties which we call 
\emph{hyperpolygonal arrangements} $\CH_n$, see Definition \ref{def:HA}.
In this note we study these
hyperpolygonal arrangements and 
investigate their properties in a systematic manner. 
It turns out that the arrangements 
$\CH_n$ differentiate essentially between all local properties of arrangements, see Theorem \ref{thm:HA}.
In addition we show that hyperpolygonal arrangements are projectively unique, see Theorem \ref{thm:HAiscomb}, and combinatorially formal, see Theorem \ref{thm:HAformal}.

We briefly indicate how the arrangements \(\CH_n\) arise in \cite{bellamyetall:crepant} and show their connection to birational geometry, see \cite{bellamyetall:crepant} for details and references.
The arrangement \(\CH_n\) characterizes stability conditions on the parameter \(\theta\) of the hyperpolygon space \(X_n(\theta)\).
Namely, \(X_n(\theta)\) is smooth if and only if \(\theta\) does not lie on any hyperplane in \(\CH_n\).
As explained in \cite{bellamyetall:crepant}, the varieties \(X_n(0)\) are \emph{conical symplectic varieties} and the map \(X_n(\theta) \to X_n(0)\) is a crepant projective resolution (hence a symplectic projective resolution), if \(\theta\) does not lie in \(\CH_n\).
Furthermore, two such resolutions \(X_n(\theta_1) \to X_n(0)\) and \(X_n(\theta_2) \to X_n(0)\) are isomorphic if \(\theta_1\) and \(\theta_2\) lie in the same region of the complement of \(\CH_n\)  in $\BBR^n$.
This construction works for all conical symplectic varieties by Namikawa \cite{Nam15}:
For a conical symplectic variety \(Y\), there is a certain ``parameter space'' containing a hyperplane arrangement that characterizes the isomorphism classes of crepant projective resolutions (or, in general, the \(\BBQ\)-factorial terminalizations) of \(Y\).
Very little is known about the hyperplane arrangements arising in this way, although they are essential in understanding the birational geometry of \(Y\), see also \cite{Bel16, BST18}.
Special cases of conical symplectic varieties are \emph{symplectic quotient singularities} \(V/G\), where \(V\) is a symplectic vector space over \(\BBC\) and \(G\leq\operatorname{Sp}(V)\) is a finite group maintaining the symplectic form.
By \cite{bellamyetall:crepant}, the hyperpolygon space \(X_n(0)\) is not a quotient singularity for \(n > 5\).
In contrast, $X_4(0)$ is the Kleinian singularity of type $D_4$, that is, the quotient of $\BBC^2$ by the quaternion group $Q_8$.  Further, 
\(X_5(0)\) is the quotient of \(\BBC^4\) by a symplectic reflection group of order 32.
This quotient was extensively studied in \cite{BS13} and \cite{DW17}.

In this note we study the hyperpolygonal arrangements $\CH_n$ in a systematic manner. 
We note that $\CH_5$
is the famous counterexample of Edelman and Reiner \cite{edelmanreiner:orlik} of Orlik's conjecture that the restriction of a free arrangement is again free. 

We first recall the definition from
\cite{bellamyetall:crepant}.
Fix $n \in \BBN$. Let $V = \BBR^n$. Let $x_1, \ldots, x_n$ be the dual basis  in $V^*$ of the standard $\BBR$-basis of $V$.
For $I \subseteq [n] = \{1, \ldots, n\}$, define the hyperplane
\[H_I 
:= \ker \left(\sum_{i \in I}x_i - \sum_{j \in [n]\setminus I}x_j\right)
\]
in $V$.

\begin{defn}
	\label{def:HA} 
	With the notation as above, the \emph{hyperpolygonal arrangement} $\CH_n$ in $V$ is defined as
	\[\CH_n := \{\ker x_i \mid i \in [n]\} \cup \{H_I \mid  \varnothing \neq I \subseteq [n]\}.
	\]
\end{defn}

For the various notions used in our main theorem, we refer the reader to Section \ref{sect:prelims}.

\begin{theorem}
	\label{thm:HA}
	Fix $n \in \BBN$. Then we have 
	\begin{itemize}
		\item[(i)]  $\CH_n$ is supersolvable if and only if $n\le 2$;
		\item[(ii)]  $\CH_n$ is inductively factored if and only if $n\le 3$;
		\item[(iii)]  $\CH_n$ is inductively free if and only if $n\le 4$;
		\item[(iv)] $\CH_n$ is free if and only if $n\le 5$;
		\item[(v)]  $\CH_n$ is simplicial if and only if $n\le 4$;
		\item[(vi)] $\CH_n$ is not $K(\pi,1)$ if $n\ge 6$.		
	\end{itemize}
\end{theorem}

It follows from Theorem \ref{thm:HA}(v) and Remark \ref{rem:kpione}(i) that $\CH_n$ is $K(\pi,1)$ for $n \le 4$ and $\CH_n$ fails to be  $K(\pi,1)$ for $n \ge 6$ by part (vi). It is not known whether $\CH_5$ is $K(\pi,1)$.

In general, $K(\pi,1)$ arrangements need not be free, e.g.~see \cite[Fig.~5.4]{orlikterao:arrangements}. However, for hyperpolygonal arrangements, this does seem to be the case.
While $\CH_5$ is free, it is not known whether $\CH_5$ is $K(\pi,1)$. We can thus formulate

\begin{corollary}
	\label{cor:kpi1}
	With the possible exception when $n=5$, 
	$\CH_n$ is $K(\pi,1)$ if and only if  $\CH_n$ is free. 
\end{corollary}

Thus with the possible exception of $\CH_5$, the hyperpolygonal arrangements 
$\CH_n$ satisfy Saito's Conjecture that for a complexified arrangement freeness implies $K(\pi,1)$. Of course, the latter is know to be false in general \cite{edelmanreiner:saito}.

Theorem \ref{thm:HAiscomb}(ii) and (v) imply the following which  
is in support of a conjecture due to Falk and Randell namely that every (complex) factored  arrangement is  $K(\pi,1)$, \cite[Probl.~3.12]{falkrandell:homotopyII}.

\begin{corollary}
	\label{cor:kpi1-2}
	If $\CH_n$ is factored, then $\CH_n$  is $K(\pi,1)$. 
\end{corollary}

A property for arrangements is said to be \emph{combinatorial} if it only depends on the intersection lattice of the underlying arrangement.
In this context our next theorem shows that the class of hyperpolygonal arrangements is
very special in the sense that essentially every property
we may formulate for members of this class is
combinatorial. This is formally captured by the notion of \emph{projective uniqueness} due to Ziegler \cite{ziegler:matroid}, see Definition \ref{def:projunique}.

\begin{theorem}
	\label{thm:HAiscomb}
	For any $n \in \BBN$,  $\CH_n$ is projectively unique.
\end{theorem}

Theorem \ref{thm:HAiscomb} implies that 
for the class of all real arrangements whose underlying matroid admits a realization over $\BBR$ as a hyperpolygonal arrangement 
freeness is combinatorial. In particular, Terao's conjecture over $\BBR$ is valid within this class, cf.~\cite[Prop.~2.3]{ziegler:matroid}. Likewise, asphericity is combinatorial within this class. Whether both these properties are combinatorial in general are longstanding and wide open problems, see  \cite[Conj.~4.138]{orlikterao:arrangements} and \cite[Prob.~3.8]{falkrandell:homotopyII}.

\bigskip

A hyperplane arrangement is called \emph{formal} provided all linear dependencies among the defining forms of the hyperplanes are generated by ones corresponding to intersections  of codimension two.
The significance of this notion stems from the fact that
complex arrangements with aspherical complements are formal, \cite[Thm.\ 4.2]{falkrandell:homotopy}.
In addition, free arrangements are known to be formal,
\cite[Cor.~2.5]{yuzvinsky:obstruction},
and factored arrangements are formal, \cite[Thm.~1.1]{moellermueckschroehre:formal}.
Thus all the properties studied in Theorem \ref{thm:HA} entail formality.  
In our next result we show that indeed  
all hyperpolygonal arrangements are \emph{combinatorially formal}, see Definition \ref{def:combformal}.

\begin{theorem}
	\label{thm:HAformal}
		For any $n \in \BBN$,  $\CH_n$ is combinatorially formal.
\end{theorem}

Theorem \ref{thm:HAformal} is proved in \S \ref{s:thm:HAformal}, based on results from \cite{moellermueckschroehre:formal}.

We end in \S \ref{s:rankgenerating} with a brief discussion of the rank generating functions of the poset of regions of the free hyperpolygonal arrangements  $\CH_n$.

For general information about arrangements
we refer the reader to
\cite{orlikterao:arrangements}.

\section{Preliminaries}
\label{sect:prelims}

\subsection{Hyperplane arrangements}
\label{ssect:arrangements}
Let $\BBK$ be a field and let 
$V = \BBK^n$
be an $n$-dimensional $\BBK$-vector space.
A \emph{hyperplane arrangement} is a pair
$(\CA, V)$, where $\CA$ is a finite collection of hyperplanes in $V$.
Usually, we simply write $\CA$ in place of $(\CA, V)$.

The \emph{lattice} $L(\CA)$ of $\CA$ is the set of subspaces of $V$ of
the form $H_1\cap \ldots \cap H_i$ where $\{ H_1, \ldots, H_i\}$ is a subset
of $\CA$.
For $X \in L(\CA)$, we have two associated arrangements,
firstly
$\CA_X :=\{H \in \CA \mid X \subseteq H\} \subseteq \CA$,
the \emph{localization of $\CA$ at $X$},
and secondly,
the \emph{restriction of $\CA$ to $X$}, $(\CA^X,X)$, where
$\CA^X := \{ X \cap H \mid H \in \CA \setminus \CA_X\}$.
The lattice $L(\CA)$ is a partially ordered set by reverse inclusion:
$X \le Y$ provided $Y \subseteq X$ for $X,Y \in L(\CA)$.

Throughout, we only consider arrangements $\CA$
such that $0 \in H$ for each $H$ in $\CA$.
These are called \emph{central}.
In that case the \emph{center}
$T(\CA) := \cap_{H \in \CA} H$ of $\CA$ is the unique
maximal element in $L(\CA)$  with respect
to the partial order.
A \emph{rank} function on $L(\CA)$
is given by $r(X) := \codim_V(X)$.
The \emph{rank} of $\CA$
is defined as $r(\CA) := r(T(\CA))$.

The \emph{Poincar\'e polynomial} 
$\pi(\CA,t) \in \BBZ[t]$ of $\CA$ is defined by 
\[
\pi(\CA,t) := \sum_{X \in L(\CA)} \mu(X)(-t)^{r(X)},
\]
and the \emph{characteristic polynomial} 
$\chi(\CA,t) \in \BBZ[t]$ 
of $\CA$ is defined by 
\[
\chi(\CA,t) := t^\ell \pi(\CA,-t^{-1}) = \sum_{X \in L(\CA)} \mu(X)t^{\dim X},
\]
where $\mu$ is the M\"obius function of $L(\CA)$, 
see \cite[Def.\ 2.48, Def.\ 2.52]{orlikterao:arrangements}.

We recall the concept of a generic arrangement from \cite[Def.~5.22]{orlikterao:arrangements}.

\begin{defn}
	\label{def:generic}
	An $\ell$-arrangement $\CA$ with $\rank(\CA)=r$ is called \emph{generic} if 
	every subarrangement $\CB$ of $\CA$ of cardinality $\ell$ is linearly independent, \cite[Def.~5.22]{orlikterao:arrangements}.
\end{defn}

\subsection{Supersolvable Arrangements}
\label{ssect:super}

Let $\CA$ be an arrangement.
Following \cite[\S 2]{orlikterao:arrangements}, we say
that $X \in L(\CA)$ is \emph{modular}
provided $X + Y \in L(\CA)$ for every $Y \in L(\CA)$, 
cf.\ \cite[Def.\ 2.32, Cor.\ 2.26]{orlikterao:arrangements}.
The following notion is due to Stanley \cite{stanley:super}. 

\begin{defn}
	\label{def:super}
	Let $\CA$ be a central (and essential) $\ell$-arrangement.
	We say that $\CA$ is \emph{supersolvable} 
	provided there is a maximal chain
	\[
	V = X_0 < X_1 < \ldots < X_{\ell-1} < X_\ell = \{0\}
	\]
	of modular elements $X_i$ in $L(\CA)$.
\end{defn}

\begin{remark}
	\label{rem:ss}
	(i).
	By \cite[Ex.\ 2.28]{orlikterao:arrangements}, 
	$V$, $\{0\}$ and the members in $\CA$ 
	are always modular in $L(\CA)$.
	It follows  that all $0$- $1$-, and $2$-arrangements are supersolvable.
	
	(ii). Supersolvability is a local property  \cite[Prop.~3.2]{stanley:super}.
\end{remark}

\subsection{Free arrangements}
\label{ssect:free}

Free arrangements play a crucial role in the theory of arrangements;
see \cite[\S 4]{orlikterao:arrangements} for the definition and
basic properties. If $\CA$ is free, then
we can associate with $\CA$ the multiset of its \emph{exponents},
denoted $\exp \CA$. 

\begin{remark}
	\label{rem:free}
	(i).
	Generic arrangements are not free, e.g.,~see \cite[\S 4.4]{roseterao}.
	
(ii). Freeness is a local property  \cite[Thm.~4.37]{orlikterao:arrangements}.
\end{remark}

Terao's \emph{Factorization Theorem}
\cite{terao:freefactors} shows
that the Poincar\'e polynomial
of a free arrangement $\CA$
factors into linear terms
given by the exponents of $\CA$
(cf.\ \cite[Thm.\ 4.137]{orlikterao:arrangements}):

\begin{theorem}
	\label{thm:freefactors}
	Suppose that
	$\CA$ is free with $\exp \CA = \{ b_1, \ldots , b_\ell\}$.
	Then
	\[
	\pi(\CA,t) = \prod_{i=1}^\ell (1 + b_i t).
	\]
\end{theorem}

Terao's celebrated \emph{Addition-Deletion Theorem}
\cite{terao:freeI} plays a
fundamental role in the study of free arrangements,
\cite[Thm.\ 4.51]{orlikterao:arrangements}.

\begin{theorem}
	\label{thm:add-del}
	Suppose that $\CA \ne \Phi_\ell$.
	Let  $(\CA, \CA', \CA'')$ be a triple of arrangements. Then any
	two of the following statements imply the third:
	\begin{itemize}
		\item[(i)] $\CA$ is free with $\exp \CA = \{ b_1, \ldots , b_{\ell -1}, b_\ell\}$;
		\item[(ii)] $\CA'$ is free with $\exp \CA' = \{ b_1, \ldots , b_{\ell -1}, b_\ell-1\}$;
		\item[(iii)] $\CA''$ is free with $\exp \CA'' = \{ b_1, \ldots , b_{\ell -1}\}$.
	\end{itemize}
\end{theorem}

Theorem \ref{thm:add-del} motivates the notion of an
\emph{inductively free} arrangement,
\cite[Def.\ 4.53]{orlikterao:arrangements}.

\begin{defn}
	\label{def:indfree}
	The class $\CIF$ of \emph{inductively free} arrangements
	is the smallest class of arrangements subject to
	\begin{itemize}
		\item[(i)] $\Phi_\ell \in \CIF$ for each $\ell \ge 0$;
		\item[(ii)] if there exists a hyperplane $H_0 \in \CA$ such that both
		$\CA'$ and $\CA''$ belong to $\CIF$, and $\exp \CA '' \subseteq \exp \CA'$,
		then $\CA$ also belongs to $\CIF$.
	\end{itemize}
\end{defn}

\begin{remark}
	\label{rem:indfree}
	(i). 
	Inductively free arrangements are free. However, the latter class properly contains the former, cf.~\cite[Ex.~4.59]{orlikterao:arrangements}.
	
	(ii). Inductive freeness is also a local property, thanks to  \cite[Thm~1.1]{hogeroehrleschauenburg:free}.
	
	(iii). Supersolvable arrangements
	are inductively free, 
	\cite[Thm.\ 4.58]{orlikterao:arrangements}.
\end{remark}

\subsection{Nice arrangements}
\label{ssect:factored}

The notion of a \emph{nice} or \emph{factored}
arrangement goes back to Terao \cite{terao:factored}.
It generalizes the concept of a supersolvable arrangement.
We recall the relevant notions and results from \cite{terao:factored}
(cf.\  \cite[\S 2.3]{orlikterao:arrangements}).

\begin{defn}
	\label{def:independent}
	Let $\pi = (\pi_1, \ldots , \pi_s)$ be a partition of $\CA$.
	Then $\pi$ is called \emph{independent}, provided
	for any choice $H_i \in \pi_i$ for $1 \le i \le s$,
	the resulting $s$ hyperplanes are linearly independent, i.e.\
	$r(H_1 \cap \ldots \cap H_s) = s$.
\end{defn}

\begin{defn}
	\label{def:indpart}
	Let $\pi = (\pi_1, \ldots , \pi_s)$ be a partition of $\CA$
	and let $X \in L(\CA)$.
	The \emph{induced partition} $\pi_X$ of $\CA_X$ is given by the non-empty
	blocks of the form $\pi_i \cap \CA_X$.
\end{defn}

\begin{defn}
	\label{def:factored}
	The partition
	$\pi$ of $\CA$ is
	\emph{nice} for $\CA$ or a \emph{factorization} of $\CA$  provided
	\begin{itemize}
		\item[(i)] $\pi$ is independent, and
		\item[(ii)] for each $X \in L(\CA) \setminus \{V\}$, the induced partition $\pi_X$ admits a block
		which is a singleton.
	\end{itemize}
	If $\CA$ admits a factorization, then we also say that $\CA$ is \emph{factored} or \emph{nice}.
\end{defn}

\begin{remark}
	\label{rem:factored}
	The class of nice arrangements is closed under taking localizations;
	cf.~the proof of \cite[Cor.~2.11]{terao:factored}.
\end{remark}

In \cite[Thm.\ 2.8]{terao:factored},
Terao proved that a partition $\pi$ of $\CA$ gives
rise to a tensor factorization of the Orlik-Solomon algebra of $\CA$
if and only if $\pi$ is
nice for
$\CA$, see \cite[Thm.~3.87]{orlikterao:arrangements}.
We record a consequence of this fact for our purposes.

\begin{corollary}
	\label{cor:teraofactored}
	Let  $\pi = (\pi_1, \ldots, \pi_s)$ be a factorization of $\CA$.
	Then the following hold:
	\begin{itemize}
		\item[(i)] $s = r = r(\CA)$ and
		\[
		\pi(\CA,t) = \prod_{i=1}^r (1 + |\pi_i|t);
		\]
		\item[(ii)]
		the multiset $\{|\pi_1|, \ldots, |\pi_r|\}$ only depends on $\CA$;
		\item[(iii)]
		for any $X \in L(\CA)$, we have
		\[
		r(X) = |\{ i \mid \pi_i \cap \CA_X \ne \varnothing \}|.
		\]
	\end{itemize}
\end{corollary}

\begin{remark}
	\label{rem:factorediscombinatorial}
	It follows from 
	Corollary \ref{cor:teraofactored} that
	the question whether $\CA$ is factored is a purely combinatorial
	property and only depends on the lattice $L(\CA)$.
\end{remark}

	Moreover, the following is immediate from
		Corollary \ref{cor:teraofactored} and Theorem \ref{thm:freefactors}.

	\begin{lemma}
		\label{lem:FactoredFree}
		Let $(\CA,\pi)$ be a factored arrangement which is also free.
		Then $\exp{\CA} = \{|\pi_1|,\ldots,|\pi_\ell|\}$.
	\end{lemma}

\subsection{Inductively factored arrangements}
\label{sec:indfactored}

Following Jambu and Paris
\cite{jambuparis:factored} and \cite{hogeroehrle:factored},
we introduce further notation.
Suppose that $\CA$ is non-empty and
let $\pi = (\pi_1, \ldots, \pi_s)$ be a  partition  of $\CA$.
Let $H_0 \in \pi_1$ and let
$(\CA, \CA', \CA'')$ be the triple associated with $H_0$.
We have the \emph{induced partition}
$\pi'$ of $\CA'$
consisting of the non-empty parts $\pi_i' := \pi_i \cap \CA'$.
Further, we have the \emph{restriction map}
$\R = \R_{\pi,H_0} : \CA \setminus \pi_1 \to \CA''$ given by
$H \mapsto H \cap H_0$, depending on $\pi$ and $H_0$.
Let $\pi_i'' := \R(\pi_i)$ for $i = 2, \ldots, s$.
Clearly, imposing that
$\pi'' = (\pi''_2, \ldots, \pi''_s)$ is
again a partition of $\CA''$ entails that
$\R$ is onto.

Here is the analogue for nice arrangements of Terao's
Addition-Deletion Theorem (cf.~Theorem \ref{thm:add-del}) for free arrangements from
\cite{hogeroehrle:factored}.

\begin{theorem}
	\label{thm:add-del-factored}
	Suppose
	$\pi = (\pi_1, \ldots, \pi_s)$ is a  partition  of $\CA  \ne \Phi_\ell$.
	Let 
	$(\CA, \CA', \CA'')$ be the triple associated with $H_0  \in \pi_1$.
	Then any two of the following statements imply the third:
	\begin{itemize}
		\item[(i)] $\pi$ is nice for $\CA$;
		\item[(ii)] $\pi'$ is nice for $\CA'$;
		\item[(iii)] $\R: \CA \setminus \pi_1 \to \CA''$
		is bijective and $\pi''$ is nice for $\CA''$.
	\end{itemize}
\end{theorem}

The Addition-Deletion Theorem \ref{thm:add-del-factored}
for nice arrangements motivates
the following stronger notion of factorization,
cf.\ \cite{jambuparis:factored}.

\begin{defn} [{\cite[Def.~3.8]{hogeroehrle:factored}}]
	\label{def:indfactored}
	The class $\CIFAC$ of \emph{inductively factored} arrangements
	is the smallest class of pairs $(\CA, \pi)$ of
	arrangements $\CA$ along with a partition $\pi$
	subject to
	\begin{itemize}
		\item[(i)] $(\Phi_\ell, (\varnothing)) \in \CIFAC$ for each $\ell \ge 0$;
		\item[(ii)] if there exists a partition $\pi$ of $\CA$
		and a hyperplane $H_0 \in \pi_1$ such that
		for the triple $(\CA, \CA', \CA'')$ associated with $H_0$
		the restriction map $\R = \R_{\pi, H_0} : \CA \setminus \pi_1 \to \CA''$
		is bijective and for the induced partitions $\pi'$ of $\CA'$ and
		$\pi''$ of $\CA''$
		both $(\CA', \pi')$ and $(\CA'', \pi'')$ belong to $\CIFAC$,
		then $(\CA, \pi)$ also belongs to $\CIFAC$.
	\end{itemize}
	If $(\CA, \pi)$ is in $\CIFAC$, then we say that
	$\CA$ is \emph{inductively factored with respect to $\pi$}, or else
	that $\pi$ is an \emph{inductive factorization} of $\CA$.
	Usually, we say $\CA$ is \emph{inductively factored} without
	reference to a specific inductive factorization of $\CA$.
\end{defn}

\begin{remark}
	\label{rem:localindfac}
	(i). 	
	If $\CA$ is inductively factored, then  $\CA$ is inductively free,
	by \cite[Prop.~3.14]{hogeroehrle:factored}.
	
	(ii).
	Thanks to \cite[Thm.~1.1]{moellerroehrle:indfac}, 
	inductive factoredness  is preserved under localizations.
	
	(iii). If $\CA$ is supersolvable, then it is inductively factored, see \cite{jambuparis:factored} or \cite[Prop.~3.11]{hogeroehrle:factored}.
\end{remark}

\subsection{Simplicial arrangements}
\label{ssect:simplicial}

A real arrangement $\CA$ is \emph{simplicial} provided each chamber of the complement of $\CA$ is an open simplicial cone in the ambient space.
Simpliciality is a combinatorial property. For, thanks to  \cite[Cor.~2.4]{cuntzgeiss}, a central essential real $\ell$-arrangement 
$\CA$ is simplicial if and only if 
\begin{equation}
	\label{eq:simplicial}
	\ell \cdot \chi(\CA, -1) + 2\sum_{H \in \CA} \chi(\CA^H, -1) = 0.
\end{equation}

\begin{remark}
	\label{rem:localsimplicial}
	Simpliciality is preserved under localizations, see \cite[Lem.~2.17(1)]{Cunzmuecksch:sssimplicial}.
\end{remark}

\subsection{$K(\pi,1)$-arrangements}
\label{ssect:kpionearrangements}
A complex $\ell$-arrangement $\CA$ is called  \emph{aspherical}, or a
\emph{$K(\pi,1)$-arrangement} (or that $\CA$ is $K(\pi,1)$ for short), provided
the complement $M(\CA)$ of the union of the hyperplanes in
$\CA$ in $\BBC^\ell$ is aspherical, i.e.~is a
$K(\pi,1)$-space. That is, the universal covering space of $M(\CA)$
is contractible and the fundamental group
$\pi_1(M(\CA))$ of $M(\CA)$ is isomorphic to the group $\pi$.
This is an important
topological property, for
the cohomology ring $H^*(X, \BBZ)$ of a $K(\pi,1)$-space $X$
coincides
with the group cohomology $H^*(\pi, \BBZ)$ of $\pi$.
The crucial point here is that the intersections of codimension $2$
determine the fundamental group $\pi_1(M(\CA))$ of $M(\CA)$.

\begin{remark}
	\label{rem:kpione}
	(i). 
	By Deligne's seminal work \cite{deligne}, complexified simplicial arrangements are $K(\pi, 1)$.
	Likewise for complex supersolvable arrangements, cf.~\cite{falkrandell:fiber-type} and \cite{terao:modular} (cf.~\cite[Prop.\ 5.12, Thm.~5.113]{orlikterao:arrangements}).
	
	(ii).
	Thanks to an observation by Oka,
	asphericity is preserved under localizations,
	e.g., see \cite[Lem.~1.1]{paris:deligne}.
	
	(iii).
		By work of Hattori, 
	generic arrangements are not $K(\pi, 1)$, \cite[Cor.~5.23]{orlikterao:arrangements}.
\end{remark}

\section{Proof of Theorem \ref{thm:HA}}
\label{s:thm:kpi1}
We begin by identifying the small rank hyperpolygonal arrangements $\CH_n$ with known ones.
Clearly, $\CH_2$ is just the reflection arrangement of the Weyl group of type $B_2$.

It follows from the next lemma that 
$\CH_3$ is linearly isomorphic to the connected subgraph arrangement $\CA_G$, where $G = C_3$ is the cycle graph on three vertices. 
See \cite{cuntzkuehne:subgraphs} for the class of 
\emph{connected subgraph arrangements}. 

\begin{lemma}
	\label{lem:H3}
	The arrangements $\CH_3$ and $\CA_{C_3}$ are linearly isomorphic.
\end{lemma}

\begin{proof}
	One checks that  (up to scalar multiples) the map on linear forms 
	$$\varphi: \BBR^3 \rightarrow \BBR^3 \ \ \ \begin{bmatrix} x_1 \\ x_2 \\ x_3 \end{bmatrix} \mapsto \begin{bmatrix} x_1+x_2 \\ x_1+x_3 \\ x_2+x_3 \end{bmatrix}$$
	gives a linear isomorphism between $\CH_3$ and $\CA_{C_3}$.
\end{proof}

We also note that $\CH_3$ is linearly isomorphic to the cone over the Shi arrangement of type  $A_2$,  by means of Lemma \ref{lem:H3} and 
\cite[Prop.~3.1]{cuntzkuehne:subgraphs}.

Our next lemma shows that $\CH_4$ is linearly isomorphic to the reflection arrangement $\CA(D_4)$ of the Weyl group of type $D_4$.

\begin{lemma}
	\label{lem:H4}
	The arrangements $\CH_4$ and $\CA(D_4)$ are linearly isomorphic.
\end{lemma}

\begin{proof}
	One checks that  (up to scalar multiples) the map on linear forms 
	$$\varphi: \BBR^4 \rightarrow \BBR^4 \ \ \ \begin{bmatrix} x_1 \\ x_2 \\ x_3  \\ x_4  \end{bmatrix} \mapsto \begin{bmatrix} x_1+x_2+x_3 \\ x_2 \\ x_2+x_3+x_4 \\ -x_1-x_2-x_4 \end{bmatrix}$$
	gives a linear isomorphism between $\CH_4$ and $\CA(D_4)$.
\end{proof}

Our next lemma gives that the reverse implications in the statements in Theorem \ref{thm:HA}  hold for 
$n\ge6$. 

\begin{lemma}
	\label{lem:genericloc}
	Let $n \ge 6$. Then there exists a generic rank $3$ localization of $\CH_n$. 	
	As a consequence, for $n \ge 6$, $\CH_n$ is not 
	free (and so is not supersolvable, not inductively free, and not inductively factored), it is also not $K(\pi,1)$ (and so is also not simplicial).  
\end{lemma}

\begin{proof}
	For $n = 6$, let $I_1 := \{1\}, I_2 := \{1, 2,3\}, I_3 := \{1, 4,5\}$, $I_4 := \{1, 2,3, 4,5\}$, and define $X = \cap_{i = 1}^4 H_{I_i} \in L(\CH_6)$. 
	Then, as the localization $(\CH_6)_X$ 
	consists of precisely the four hyperplanes $H_{I_1}, \dots, H_{I_4}$ of $\CH_6$, we infer that $(\CH_6)_X$
	is generic of rank $3$, see \cite[Ex.~4.5.6]{roseterao}. 
	
	Generalizing this example for $n > 6$,  for $X = \cap_{i = 1}^4 H_{I_i} \in L(\CH_n)$, the very same argument shows that
	also the localization $(\CH_n)_X$ in $\CH_n$ is still generic of rank $3$.  

	It follows from Remark \ref{rem:free}(i) that $(\CH_n)_X$ is not free and from Remark \ref{rem:kpione}(iii) that $(\CH_n)_X$ is not $K(\pi,1)$. Consequently, $(\CH_n)_X$ is not supersolvable, not inductively free, not inductively factored and not simplicial. And as all of the latter are local properties, thanks to Remarks \ref{rem:ss}(ii),   \ref{rem:free}(ii), \ref{rem:indfree}(ii), \ref{rem:localindfac}(ii), \ref{rem:localsimplicial}, and \ref{rem:kpione}(ii), the lemma follows.
\end{proof}

Thus it remains to show the forward implications in the statements in Theorem \ref{thm:HA} and the remaining reverse implications for $n \le 5$.

Part (i).
Thanks to Remark \ref{rem:ss}(i), $\CH_2$ is supersolvable. 
It follows from Lemma \ref{lem:H3} and  
\cite[Cor.~8.11]{cuntzkuehne:subgraphs} that $\CH_3$ 
is not supersolvable. 
By \cite[Lem.~3.2]{hogeroehrle:nice}, $\CA(D_4)$ is not factored (thus not inductively factored), thus, thanks to Lemma \ref{lem:H4}, neither is $\CH_4$, so the latter is not supersolvable, by  Remark \ref{rem:localindfac}(iii). Thanks to  \cite{edelmanreiner:orlik}, $\CH_5$ is not inductively free, thus by Remark \ref{rem:indfree}(iii) it is not supersolvable. 
Thus Theorem \ref{thm:HA}(i) now follows from Lemma \ref{lem:genericloc}.

Part (ii).
It follows from  \cite[Thm.~1.7]{GMMR} that $\CA_{C_3}$ is inductively factored, thus so is $\CH_3$, by Lemma \ref{lem:H3}.
We have already observed above that $\CH_4$ is not inductively factored. 
Since $\CH_5$ is not inductively free, it is also not inductively factored, by Remark \ref{rem:localindfac}(i).
Consequently, Theorem \ref{thm:HA}(ii) follows again from Lemma \ref{lem:genericloc}.

Part (iii).
For $n \le 4$, the arrangements $\CH_n$ are inductively free: 
For $n \le 3$, this follows from Theorem \ref{thm:HA}(ii) and Remark \ref{rem:localindfac}(i).
	For $\CH_4$ this follows from Lemma \ref{lem:H4} and \cite[Ex.~2.6]{jambuterao:free}.
Edelman and Reiner \cite{edelmanreiner:orlik} have observed that 
$\CH_5$ is free but not inductively free.
Consequently, Theorem \ref{thm:HA}(iii) follows from Lemma \ref{lem:genericloc}.
	
Part (iv).
By  Theorem \ref{thm:HA}(iii) and \cite{edelmanreiner:orlik},  $\CH_n$ is free for $n \le 5$. 
Whence Theorem \ref{thm:HA}(iv) is once again a consequence of Lemma \ref{lem:genericloc}. 

Part (v).
With the aid of the formula \eqref{eq:simplicial} it is straightforward to 
check that $\CH_n$ is simplicial for $n \le 4$ and that $\CH_5$ is no longer simplicial. So Theorem \ref{thm:HA}(v) follows from Lemma \ref{lem:genericloc}.

Finally, Theorem \ref{thm:HA}(vi) is immediate from Lemma \ref{lem:genericloc}.

In closing this section 
we show that $\CH_5$ is still \emph{recursively free}, a concept originally due to Ziegler \cite{ziegler:phd}; see  
\cite[Def.~4.60]{orlikterao:arrangements}.

\begin{remark}
	\label{rem:h5recfree}
	While $\CH_5$ is not inductively free, one can show that it is still recursively free. We sketch the argument. 
	First one checks that $\CA = \CH_5 \cup \{\ker(x_2 - x_4)\}$ is inductively free  with $\exp(\CA) = \{1,5,5,5,6\}$. Next one checks that $\CB := \CA^{\ker(x_2 - x_4)}$ is itself again recursively free with $\exp(\CB) = \{1,5,5,5\}$. Owing to \cite[Def.~4.60]{orlikterao:arrangements}, $\CH_5$ is recursively free.
	To see in turn that $\CB$ is recursively free, one first checks that  $\CC := \CB \cup \{\ker(x_1 - x_3)\}$ is inductively free with $\exp(\CC) = \{1,5,5,6\}$. Then one checks that 
	$\CD := \CC^{\ker(x_1 - x_3)}$ is free with $\exp \CD = \{1,5,5\}$. Thanks to  \cite[Thm.~1.1]{abeetall:free}, $\CD$ is recursively free. Thus, again by  \cite[Def.~4.60]{orlikterao:arrangements},  $\CB$ is recursively free.
	
	It thus follows from the paragraph above about $\CH_5$ and Theorem \ref{thm:HA}(iv) that $\CH_n$ is free if and only if $\CH_n$ is  recursively free. In general these two notions differ, see \cite{cuntzhoge}.  
\end{remark}

\section{Projective Uniqueness: Proof of Theorem \ref{thm:HAiscomb}}
\label{s:thm:HAiscomb}

\begin{defn}
	\label{def:projunique}
	Let $\CA$, $\CB$ be two arrangements in a finite-dimensional $\BBR$-vector space $V$.
	\begin{itemize}
		\item[(i)] $\CA$ and $\CB$ are \emph{linearly isomorphic} if there is a $\varphi \in \GL(V)$
		such that $\CB =  \{\varphi(H) \mid H \in \CA\}$; denoted by $\CA \cong \CB$.
		\item[(ii)] $\CA$ and $\CB$ are \emph{$L$-equivalent} if $L(\CA)$ and $L(\CB)$ are isomorphic as posets;
		denoted by $\CA \cong_L \CB$.
		\item[(iii)] $\CA$ is \emph{projectively unique} if for any arrangement $\CC$ in $V$ we have:
		$\CC \cong_L \CA$ implies $\CC \cong \CA$.
	\end{itemize}
\end{defn}

We recall some results from \cite{GMMR}.
Let $\CA$ be an $\BBR$-arrangement.
We specify what we mean by a subarrangement of $\CA$ being generated by a subarrangement of $\CA$. 

\begin{defn}
	\label{def:gen}
	Let $\varnothing \ne S \subseteq \CA$. Set $\Gen_0(\CA,S) := S$ and inductively $$\Gen_{i+1}(\CA,S) := \left\{H \in \CA \mid \exists\  J \subseteq L(\Gen_i(\CA,S)) : H = \sum_{X \in J} X \right\}$$ for $i \ge 0$.
	Then we say that
	$$\langle S \rangle_\CA := \bigcup_{i\ge 0} \Gen_i(\CA,S) \subseteq \CA$$
	is the subarrangement of $\CA$
	\emph{generated by $S$}. If $\langle S \rangle_\CA  = \CA$, then we say that $S$ generates $\CA$. 
\end{defn}

\begin{lemma}[{\cite[Lem.~3.5]{GMMR}}]
	\label{lem:gen}
	Let $\CA$ and $\CB$ be two arrangements in $V$.
	Suppose $\varnothing \ne S \subseteq \CA$, $\varnothing \ne T \subseteq \CB$
	such that $\langle S \rangle_\CA = \CA$ and $\langle T \rangle_\CB = \CB$, i.e.~$S$ generates $\CA$
	and $T$ generates $\CB$.
	If $\CA$ and $\CB$ are $L$-equivalent via a poset isomorphism $\psi:L(\CA)\to L(\CB)$ and $\varphi \in \GL(V)$
	such that 
	$\psi(S) = \varphi(S) = T$ and $\psi(H) = \varphi(H)$ for all $H \in S$, 
	then $\varphi$ extends to a linear isomorphism between the whole arrangements, i.e.\
	$\varphi(\CA) = \CB$.
\end{lemma}

The following is a consequence of Lemma \ref{lem:gen}.

\begin{prop}[{\cite[Prop.~3.6]{GMMR}}]
	\label{prop:gen}
	Let $\CA$ be an essential and irreducible arrangement in $V \cong \BBR^\ell$.
	Suppose there is a subset $S$ of $\CA$ such that $\langle S \rangle_\CA = \CA$ and $|S| = \ell+1$.
	Then $\CA$ is projectively unique.
\end{prop}

Finally, the following result gives
Theorem \ref{thm:HAiscomb}.

\begin{prop}
	\label{prop:genCSG}
	Let $ n \ge 1$ and  
	let 
	$$S := \{\ker x_1, \ldots, \ker x_{n-1}, \ker(x_1 + \ldots + x_n), \ker(x_1 + \ldots + x_{n-1} - x_n)\} \subseteq \CH_n.$$ Then
	$\langle S \rangle_{\CH_n} = \CH_n$.
	In particular, $\CH_n$ is projectively unique over $\BBR$.
\end{prop}

\begin{proof}
	Let $\varnothing \ne I \subseteq [n]$. Without loss, we may assume that $n \notin I$. For $i \in [n-1]$, set $H_i := \ker x_i$, and 
	for $\alpha := x_1 + \ldots + x_n$ and $\beta := x_1 + \ldots + x_{n-1} -  x_n$, set $H_\alpha := \ker \alpha$ and $H_\beta := \ker \beta$. 
	Then 
	$H_I = \ker \left(\sum_{i \in I}x_i - \sum_{j \in [n]\setminus I}x_j\right) = \ker \gamma$, where 
	\begin{equation}
		\label{eq:gamma}
		\gamma = - \alpha + 2 \sum_{i \in I}x_i = \beta - 2 \sum_{j \in [n-1]\setminus I}x_j.
	\end{equation}
	It follows from \eqref{eq:gamma} that 
		\begin{align*}
		X & := H_\alpha \cap \bigcap_{i \in I} H_i \subseteq H_I \text{ and } \\
		Y & :=  H_\beta \cap \bigcap_{j \in [n-1]\setminus I} H_j \subseteq H_I.
	\end{align*}
	Since $\dim X = n - |I| - 1$, $\dim Y = |I|$, and  $\dim (X \cap Y) = 0$, we have $H_I = X + Y$.
	Since both $X$ and $Y$ belong to the lattice of intersections of the subarrangement $S$ of $\CH_n$, we infer $H_I = X + Y \in \langle S \rangle_{\CH_n}$. 
	Consequently, the result follows thanks to Proposition \ref{prop:gen}.
	\end{proof}

\section{Combinatorial Formality: Proof of Theorem \ref{thm:HAformal}}
\label{s:thm:HAformal}

A property for arrangements is said to be \emph{combinatorial} if it only depends on the intersection lattice of the underlying arrangement.
Yuzvinsky \cite[Ex.~2.2]{yuzvinsky:obstruction} demonstrated that formality is not combinatorial, answering a question raised by Falk and Randell \cite{falkrandell:homotopy} in the negative.
Yuzvinsky's insight motivates the following notion from \cite{moellermueckschroehre:formal}.

\begin{defn}
	\label{def:combformal}
	Suppose $\CA$ is a formal arrangement. We say $\CA$ is \emph{combinatorially formal} if every arrangement with 
	an intersection lattice isomorphic to the one of
	$\CA$ is also formal.
\end{defn}

The following definitions, which are originally due to Falk  for matroids \cite{falk:line-closure}, were adapted for arrangements in \cite[\S 2.4]{moellermueckschroehre:formal}.
Let $\CB \subset \CA$ be a subset of hyperplanes.
We say $\CB$ is \emph{closed} if $\CB =\CA_Y$ for $Y=\bigcap\limits_{H\in \CB} H$.
We call $\CB $ \emph{line-closed}
if for every pair $H,H'\in \CB $ of hyperplanes, we have $\CA_{H\cap H'}\subset \CB $.
The \emph{line-closure} $\lc(\CB)$ of $\CB $ is defined
as the intersection of all line-closed subsets of $\CA$ containing $\CB$.
The arrangement $\CA$ is called \emph{line-closed}
if every line-closed subset of $\CA$ is closed.
With these notions, we have the following criterion for combinatorial formality, see \cite[Cor.~3.8]{falk:line-closure}, \cite[Prop.~3.2]{moellermueckschroehre:formal}:

\begin{prop}
	\label{prop:lcbasis}
	Let $\CA$ be an arrangement of rank $r$. Suppose $\CB \subseteq \CA$ consists of $r$ hyperplanes such that $r(\CB)=r$ and $\lc(\CB)=\CA$. Then $\CA$ is combinatorially formal.
\end{prop}

A subset $\CB \subseteq \CA$ as in Proposition \ref{prop:lcbasis} is called an \emph{lc-basis} of $\CA$.

\begin{proof}[{Proof of Theorem \ref{thm:HAformal}}]
Let $\CB = \{\ker x_1, \ldots, \ker x_{n-1}, \ker(x_1 + \ldots + x_n)\} \subseteq \CH_n$.
Then it is easy to see that successively all $H_I$ for $\varnothing \ne I \subseteq [n]$ belong to the line-closure $\lc(\CB)$ of $\CB$, as follows. 
For $i \in [n]$, set $H_i := \ker x_i$, and 
for $\alpha := x_1 + \ldots + x_n$ set $H_\alpha := \ker \alpha$. 
Then for $I = [n]\setminus \{i\}$, we have
$H_I = \ker \left(\alpha -2 x_i\right) \supset H_\alpha \cap H_i$. Thus for each $I$ of cardinality $n-1$, $H_I$ belongs to $\lc(\CB)$. Now iterate the argument above for $I$ of successively smaller cardinality.
Consequently, $\lc(\CB) =  \CH_n$. Since
	$\rk(\CH_n) = n = |\CB| = \rk(\CB)$, it follows from Proposition \ref{prop:lcbasis}	that $\CH_n$ is combinatorially formal.
\end{proof}

There is a stronger notion of formality for an arrangement $\CA$, that of \emph{$k$-formality} for $1 \le k \le \rk(\CA)$ due to Brandt and Terao \cite{brandtterao}.
In view of Theorem \ref{thm:HAformal} and in view of the fact that all free arrangements are not just formal but are $k$-formal for all $k$, by \cite[Thm.~4.15]{brandtterao}, one might ask for this stronger notion of $k$-formality among  hyperpolygonal arrangements.
Computational evidence for further non-free  hyperpolygonal arrangements suggests the following.

\begin{conjecture}
	For any $n \in \BBN$, $\CH_n$ is $k$-formal for any $k$.
\end{conjecture}

\section{Rank-generating functions of the poset of regions}
\label{s:rankgenerating}

Let $\CA$ be a 
hyperplane arrangement in the real vector space $V=\BBR^\ell$. 
A \emph{region} of $\CA$ is a connected component of the 
complement $M(\CA) := V \setminus \cup_{H \in \CA}H$ of $\CA$.
Let $\RR := \RR(\CA)$ be the set of regions of $\CA$.
For $R, R' \in \RR$, we let $\CS(R,R')$ denote the 
set of hyperplanes in $\CA$ separating $R$ and $R'$.
Then with respect to a choice of a fixed 
base region $B$ in $\RR$, we can partially order
$\RR$ as follows:
\[
R \le R' \quad \text{ if } \quad \CS(B,R) \subseteq \CS(B,R').
\]
Endowed with this partial order, we call $\RR$ the
\emph{poset of regions of $\CA$ (with respect to $B$)} and denote it by
$P(\CA, B)$. This is a ranked poset of finite rank,
where $\rk(R) := |\CS(B,R)|$, for $R$ a region of $\CA$, 
\cite[Prop.\ 1.1]{edelman:regions}.
The \emph{rank-generating function} of $P(\CA, B)$ is 
defined to be the following polynomial in 
$\BBZ_{\ge 0}[t]$
\begin{equation*}
	\label{eq:rankgen}
	\zeta(P(\CA,B); t) := \sum_{R \in \RR}t^{\rk(R)}. 
\end{equation*}
This poset along with its rank-generating function
was introduced by Edelman 
\cite{edelman:regions}.

Thanks to work of  Bj\"orner, Edelman, and Ziegler 
\cite[Thm.~4.4]{bjoerneredelmanziegler}
(see also Paris \cite{paris:counting}), respectively  
Jambu and Paris \cite[Prop.~3.4, Thm.~6.1]{jambuparis:factored},
in case of a real arrangement $\CA$  
which is supersolvable, 
respectively inductively factored, 
there always exists a suitable base region $B$ so that 
$\zeta(P(\CA,B); t)$
admits a multiplicative decomposition which 
is determined by the exponents of $\CA$, i.e.
	\begin{equation}
	\label{eq:poinprod}
	\zeta(P(\CA,B); t) = \prod_{i=1}^\ell (1 + t + \ldots + t^{e_i}),
\end{equation}
	where $\{e_1, \ldots, e_\ell\} = \exp \CA$ is the 
set of exponents of $\CA$.

Quite remarkably
many classical real arrangements do satisfy the
factorization identity \eqref{eq:poinprod}, 
the most prominent ones being Coxeter arrangements.

Let $W = (W,S)$ be a Coxeter group with associated reflection arrangement 
$\CA = \CA(W)$ which consists of the reflecting hyperplanes of 
the reflections in $W$ in the real space $V=\BBR^n$, where $|S| = n$. 
The 
\emph{Poincar\'e polynomial} $W(t)$ of 
the Coxeter group $W$ is the polynomial in $\BBZ[t]$ defined by 
\begin{equation}
	\label{eq:poncarecoxeter}
	W(t) := \sum_{w \in W} t^{\ell(w)},
\end{equation}
where $\ell$ is the length function 
of $W$ with respect to $S$.
Then $W(t)$
coincides with the rank-generating function of the poset of regions $\zeta(P(\CA,B); t)$ of 
the underlying reflection arrangement 
$\CA = \CA(W)$ with respect to $B$ being the dominant Weyl chamber of $W$ in $V$; 
see \cite{bjoerneredelmanziegler} or \cite{jambuparis:factored}.

The following factorization of 
$W(t)$ is due to Solomon \cite{solomon:chevalley}:
\begin{equation}
	\label{eq:solomon}
	W(t) = \prod_{i=1}^n(1 + t + \ldots + t^{e_i}),
\end{equation}
where $\{e_1, \ldots, e_n\}$ is the 
set of exponents of $W$, i.e., the set of exponents of $\CA(W)$; see also \cite{macdonald:coxeter}. So by the comments above, \eqref{eq:solomon} coincides with the factorization in \eqref{eq:poinprod}.

Let $W$ be a Coxeter group again with reflection arrangement $\CA = \CA(W)$, 
let $X$ be a member of the intersection lattice $L(\CA)$, and consider the restricted reflection arrangement $\CA^X$. In general, $\CA^X$ is no longer a reflection arrangement. 
It was shown in \cite[Thm.~1.3]{moellerroehrle:nice}
that there always exists a suitable base region $B$ of $\CA^X$ in $X$ so that also 
$\zeta(P(\CA^X,B); t)$ satisfies \eqref{eq:poinprod}, with the exception of only three instances when $W$ is of type $E_8$.

Moreover, also the rank-generating function of the poset of regions $\zeta(P(\CA_\CI,B); t)$ for a so called \emph{ideal arrangement} $\CA_\CI$ also obeys the factorization identity \eqref{eq:poinprod}.
Ideal arrangements $\CA_\CI$ stem from ideals $\CI$ in the poset of positive roots associated to a Weyl group; see 
\cite{sommerstymoczko}, 
\cite{abeetall:weyl},
\cite{roehrle:ideal}, and \cite{abeetal:hess}.

We close with a comment on the
rank-generating function of the poset of regions 
of the free hyperpolygonal arrangements $\CH_n$.

\begin{remark}
	\label{rem:Hnzeta}
	It follows from Theorem \ref{thm:HA}(ii) and \cite[Prop.~3.4, Thm.~6.1]{jambuparis:factored} that 
	$\CH_n$ satisfies \eqref{eq:poinprod} for $n \le 3$, and for $\CH_4$ this follows from 
	Lemma \ref{lem:H4} and \eqref{eq:solomon}. One can check that in contrast \eqref{eq:poinprod} fails for $\CH_5$.
\end{remark}


\bigskip

\addcontentsline{toc}{section}{Acknowledgments}

\noindent {\bf Acknowledgments}: 
We are grateful to T.~Hoge and S.~Wiesner
for helpful discussions on material in this paper.
The research of this work was supported in part by
the DFG (Grant \#RO 1072/25-1 (project number: 539865068) to G.~R\"ohrle).


\bigskip

\bibliographystyle{amsalpha}

\newcommand{\etalchar}[1]{$^{#1}$}
\providecommand{\bysame}{\leavevmode\hbox to3em{\hrulefill}\thinspace}
\providecommand{\MR}{\relax\ifhmode\unskip\space\fi MR }
\providecommand{\MRhref}[2]{%
  \href{http://www.ams.org/mathscinet-getitem?mr=#1}{#2} }
\providecommand{\href}[2]{#2}


\end{document}